\documentclass[a4paper,12pt]{article}

\usepackage[latin1]{inputenc}
\usepackage[T1]{fontenc}
\usepackage{setspace}
\usepackage[english]{babel}
\usepackage{amsmath}
\usepackage{amsthm}
\usepackage{amssymb}
\usepackage{mathbbol}
\usepackage{graphicx}
\usepackage{bbm}

\newtheorem{theorem}{Theorem}[section]
\newtheorem{lemma}[theorem]{Lemma}

\newtheorem{remark}[theorem]{Remark}

\begin{document}
\title{A norm-inequaltity related to affine regular hexagons}
\author{\normalsize{Reinhard Wolf}\\ \small{Fachbereich Mathematik, Universit\"at Salzburg,}\\
\small{Hellbrunnerstraße 34, A-5020 Salzburg, Austria}\\ \footnotesize{e-mail address: reinhard.wolf@sbg.ac.at}}
\date{}
\maketitle

\begin{abstract}
Let $(E, \lVert . \rVert)$ be a two-dimensional real normed space with unit sphere
$S = \{x \in E, \lVert x \rVert = 1\}$. The main result of this paper is the following:\\
Consider an affine regular hexagon with vertex set $H = \{\pm v_1, \pm v_2, \pm v_3\} \subseteq S$ inscribed to $S$. Then we have
$$\min_i \max_{x \in S}{\lVert x - v_i \rVert + \lVert x + v_i \rVert} \leq 3.$$
From this result we obtain
$$\min_{y \in S} \max_{x \in S}{\lVert x - y \rVert + \lVert x + y \rVert} \leq 3,$$
and equality if and only if $S$ is a parallelogram or an affine
regular hexagon.
\end{abstract}
\section{Introduction}
Let $(E,\lVert . \rVert)$ be a two-dimensional real normed space with unit sphere \\
$S = \{x \in E, \lVert x \rVert = 1 \}$. This paper studies the function
$$f: S \rightarrow \mathbbm{R}, \quad f(y) = \max_{x \in \text{S}}{\lVert x-y \rVert + \lVert x+y \rVert}, \qquad y \text{ in S}.$$
In particular, we are interested in upper bounds for $\min_{y \in \text{S}}{f(y)}$.\\
The following two examples are of special interest (see Theorem 3.2):\\
\begin{enumerate}
    \item Let $S$ be a parallelogram, i.e. $(E, \lVert . \rVert)$ is isometrically isomorphic to $\mathbbm{R}^2$ equipped with the usual 1-norm.\\
    It is easy to check, that $f(S) = [3,4]$ and therefore
    $$\min_{y \in \text{S}}{f(y)} = 3.$$
    Further note, that $f(y) = 3$ if and only if $y \in \{\pm (\frac{1}{2},\frac{1}{2}), \pm (\frac{1}{2},-\frac{1}{2})\}$.\\
    So up to isometries there exists exactly one point y in $S$, such that $f(y) = 3$.

    \item Let $S$ be an affine regular hexagon (the affine image of a Euclidian equilateral hexagon). Routine calculations show, that $f(S) = \{3\}$ and therefore
    $$\min_{y \in \text{S}}{f(y)} = 3.$$
    In contrast to the first example all points y in S have the property that $f(y) = 3$.
\end{enumerate}
The main result of this paper is the following (Theorem 3.1):\\\newline
Consider an affine regular hexagon with vertex set $H = \{\pm v_1, \pm v_2, \pm v_3 \} \subseteq S$ inscribed to $S$. Then we have
$$\min (f(v_1),f(v_2),f(v_3)) \leq 3.$$
From this result we obtain an upper bound for $\min_{y \in \text{S}}{f(y)}$, namely we show (Theorem 3.2):\\
$\min_{y \in \text{S}}{f(y)} \leq 3$ and equality if and only if $S$ is a parallelogram or an affine regular hexagon.\\
\newline
This estimate is an improvement of a result given by M. Baronti, E.
Casini and P.L. Papini (see Proposition 2.8 in \cite{1}):\\
They showed, that
$$\min_{y \in \text{S}}{f(y)} \leq \frac{1 + \sqrt{1 + 4p}}{2},$$
where $p$ denotes the perimeter (measured by the norm) of $S$.\\
It is well known, that $6 \leq p \leq 8$ (for example see Satz 11.9
in \cite{2}) and hence
$$3 \leq \frac{1 + \sqrt{1 + 4p}}{2}.$$
We end this section with some well known facts about affine regular hexagons inscribed to the unit sphere $S$:\\
\newline
Fix some point $v_1$ in $S$. Since the function $x \mapsto \lVert x - v_1 \rVert$ is continuous on $S$ and $\lVert v_1 - v_1 \rVert = 0, \lVert (-v_1) - v_1 \rVert = 2$, we find some $v_2$ on $S$ (going from $v_1$ to $-v_1$ in counter-clockwise direction), such that $\lVert v_2 - v_1 \rVert = 1$.\\
With $v_3 = v_2 - v_1$ we obtain an affine regular hexagon with vertex set $H = \{\pm v_1, \pm v_2, \pm v_3 \} \subseteq S$ inscribed to $S$.\\
On the other hand it is easy to see, that an affine regular hexagon with vertex set $H = \{\pm v_1, \pm v_2, \pm v_3 \} \subseteq S$ inscribed to $S$ (the arrangement of the vertices is assumed in counter-clockwise direction: $v_1, v_2, v_3, -v_1, -v_2, -v_3, v_1$) has the property, that $v_3 = v_2 - v_1$.\\
So in the sequel an affine regular hexagon with vertex set \\
$H = \{\pm v_1, \pm v_2, \pm v_3 \} \subseteq S$ inscribed to $S$ is
given by
\begin{itemize}
    \item a fixed point $v_1$ in $S$
    \item a point $v_2$ in S with $\lVert v_2 - v_1 \rVert = 1$, found by going from $v_1$ to $-v_1$ in counter-clockwise direction
    \item $v_3 = v_2 - v_1$
\end{itemize}
\section{Notation}
Let $(E, \lVert . \rVert)$ be a two-dimensional real normed space.\\
The unit sphere of E is denoted by $S, \quad S = \{x \in E, \lVert x \rVert = 1 \}$.\\
For $x,y$ in $E$ the closed (straight line) segment from $x$ to $y$ is denoted by $\overline{xy}, \quad \overline{xy} = \{(1 - \lambda)x + \lambda y, \ 0 \leq \lambda \leq 1 \}$.\\
For $x,y$ in $S$ $(y \neq -x)$ the closed (shorter) arc joining $x$ and $y$ is defined as $[x,y], \quad [x,y] = \{\lambda x + \mu y, \ \lambda, \mu \geq 0\} \cap S$.\\
Furthermore $(x,y] = [x,y]\setminus \{x\}, \quad [x,y) = [x,y]\setminus \{y\}$ and \\
$(x,y) = [x,y]\setminus \{x, y\}$.\\
The orientation of $S$ (considered as a closed curve) is always assumed to be counter-clockwise:\\
If we say $v_1, v_2, \dotsc, v_n$ are points on $S$ or defining a subset $\{v_1, v_2, \dotsc, v_n \}$ of $S$, we assume, that a walk on $S$ in counter-clockwise direction starting in $v_1$ first
 reaches $v_2$, then $v_3, \dotsc $, then $v_{n-1}$ and ends in $v_n$.\\
The notation $\{\pm v_1, \pm v_2, \dotsc,\pm v_n\}$ is used for the set \\
$\{v_1, v_2, \dotsc, v_n,-v_1,- v_2, \dotsc,- v_n \} \subseteq S$.
\section{The results}
\begin{theorem}
    Let $(E, \lVert . \rVert)$ be a two-dimensional real normed space with unit sphere $S = \{x \in E, \lVert x \rVert = 1\}$.\\
    Consider an affine regular hexagon with vertex set $H = \{\pm v_1, \pm v_2, \pm v_3 \} \subseteq S$ inscribed to S. Then we have
    $$\min_i{\max_{x \in \text{S}}{\lVert x - v_i \rVert + \lVert x + v_i \rVert}} \leq 3.$$
\end{theorem}
\begin{theorem}
    Let $(E, \lVert . \rVert)$ be a two-dimensional real normed space with unit sphere $S = \{x \in E, \lVert x \rVert = 1\}$. Then we have
    $$\min_{y \in \text{S}}{\max_{x \in \text{S}}{\lVert x - y \rVert + \lVert x + y \rVert}} \leq 3$$
    and equality if and only if S is a parallelogram or an affine regular hexagon.
\end{theorem}
\begin{remark}
Let $(E, \left\|.\right\|)$ be a n-dimensional real normed space with unit sphere
$S = \left\{x \in E, \left\|x\right\| = 1\right\}$.
It is easy to check, that for
\begin{eqnarray*}
(E, \left\|.\right\|) = (\mathbb{R}^n, \left\|.\right\|_1)
\end{eqnarray*}
($\left\|.\right\|_1$ denotes the usual 1-norm) we get
\begin{eqnarray*}
\min_{y \in S}\max_{x \in S} \left\|x - y\right\| + \left\|x + y\right\| = 4 - \frac{2}{n}.
\end{eqnarray*}
Furthermore recall the well known fact, that each two-dimensional
real normed space is $L^1$-embeddable, i.e. isometrically isomorphic
to a subspace of $L^1\left[0,1\right]$.\\\newline We conjecture (at
least for $L^1$-embeddable) n-dimensional real normed spaces $(E,
\left\|.\right\|)$, that
\begin{eqnarray*}
\min_{y \in S}\max_{x \in S} \left\|x - y\right\| + \left\|x + y\right\| \leq 4 - \frac{2}{n}
\end{eqnarray*}
holds.
\end{remark}
\section{The proofs}
First we recall the following result:\\\newline
Fix some $v_1$ on the unit sphere $S$ of a two-dimensional real normed space $(E, \lVert . \rVert)$. The value $\lVert x - v_1 \rVert$ is non decreasing as $x$ moves on the unit sphere from $v_1$ to $-v_1$. This result is known as the so called monotonicity lemma.\\\newline
A generalization of the monotonicity lemma is given by
\begin{lemma}
    {\normalfont(= Proposition 31 in \cite{3})}\\
    Let $(E, \lVert . \rVert)$ be a two-dimensional real normed space with unit sphere \\
    $S = \{x \in E, \lVert x \rVert = 1\}$.\\
    Let $x_1, x_2, x_3 \neq 0, \ x_1 \neq x_3$, such that the halfline $\{\lambda x_2, \lambda \geq 0\}$ lies between the halflines $\{\lambda x_1, \lambda \geq 0\}$ and $\{\lambda x_3, \lambda \geq 0\}$, and suppose that $\lVert x_2 \rVert = \lVert x_3 \rVert$.\\
    Then $\lVert x_1 - x_2 \rVert \leq \lVert x_1 - x_3 \rVert$, with equality if and only if either
    \begin{enumerate}
        \item $x_2 = x_3$
        \item or $0$ and $x_2$ are on opposite sides of the line through $x_1$ and $x_3$, and $\overline{(x_3 - x_1)/\lVert x_3 - x_1 \rVert \ x_2/\lVert x_2 \rVert}$ is a segment on S,
        \item or $0$ and $x_2$ are on the same side of the line through $x_1$ and $x_3$, and $\overline{(x_3 - x_1)/\lVert x_3 - x_1 \rVert \ (-x_3)/\lVert x_3 \rVert}$ is a segment on S.
    \end{enumerate}
\end{lemma}
\begin{lemma}
    Let $(E, \lVert . \rVert)$ be a two-dimensional real normed space with unit sphere $S = \{x \in E, \lVert x \rVert = 1\}$.\\
    Consider an affine regular hexagon with vertex set $H = \{\pm v_1, \pm v_2, \pm v_3 \} \subseteq S$ inscribed to S. Further let $x$ be in $[v_1, v_2]$. \\
    Then we have $\lVert x - v_1 \rVert + \lVert x + v_1 \rVert \leq 3$, and equaltity if and only if either
    \begin{enumerate}
        \item $x = v_2$ and $[v_1, v_2] = \overline{v_1 v_2}$ or

        \item $x \in (v_1, v_2)$ and $[v_1, x] = \overline{v_1 x}, \ [x, v_3] = \overline{x v_3}$
    \end{enumerate}
    An analogous result holds for $x$ in $[v_3, -v_1]$.\\
\end{lemma}
\begin{proof} Of course we can assume, that $x$ is in $(v_1, v_2]$:
\begin{itemize}
    \item $x = v_2$ leads to
    $$\lVert x - v_1 \rVert + \lVert x + v_1 \rVert = 1 + \lVert v_2 + v_1 \rVert \leq 3$$
    and equality if and only if $\lVert v_2 + v_1 \rVert = 2$, but then we get $[v_1, v_2] = \overline{v_1 v_2}$.

    \item $x \in (v_1, v_2)$\\
    For $x_1 = v_1, x_2 = x$ and $x_3 = v_2$ Lemma 4.1 shows, that \\
    $\lVert v_1 - x \rVert \leq \lVert v_1 - v_2 \rVert = 1$ and hence $\lVert x - v_1 \rVert + \lVert x + v_1 \rVert \leq 3$.\\
    \newline
    If $\lVert x - v_1 \rVert + \lVert x + v_1 \rVert = 3$ we get $\lVert v_1 - x \rVert = 1$ and $\lVert x + v_1 \rVert = 2$.\\
    $\lVert x + v_1 \rVert = 2$ leads to $[v_1, x] = \overline{v_1 x}$ and by Lemma 4.1, part 2 we have $[x v_3] = \overline{x v_3}$.
\end{itemize}
\end{proof}

\begin{lemma}
     Let $a_1, a_2, a_3, b_1, b_2, b_3$ be real numbers and \\
     set $a_7 = a_4 = a_1, a_6 = a_3, a_5 = a_2, b_7 = b_4 = b_1, b_6 = b_3$ and $b_5 = b_2$.\\
     Assume that $0 < a_i < 1, \ 0 < b_i < 1$ and $a_i + b_{i+2} \leq 1$, for all $i = 1,2,3$.\\
     Further for $i = 1,2,3$ let
     $$s_i = -\frac{a_{i+1}}{a_i} - \frac{b_{i+2}}{b_i} + a_{i+1} + b_{i+2} + 1$$
     $$t_i = -\frac{a_{i+2}}{b_i} - \frac{a_{i+2}}{a_i} - \frac{a_{i+1}}{a_i} + a_{i+1} + 1$$
     $$u_i = -\frac{b_{i+2}}{b_i} - \frac{b_{i+1}}{b_i} - \frac{b_{i+1}}{a_i} + b_{i+2} + 1$$
     $$v_i = -\frac{a_{i+2}}{b_i} - \frac{b_{i+1}}{b_i} - \frac{a_{i+2}}{a_i} - \frac{b_{i+1}}{a_i} + 1$$
     $$(s_5 = s_2, s_4 = s_1, \dotsc , v_5 = v_2, v_4 = v_1)$$
     Then we have
     \begin{enumerate}
        \item $\min{(s_1,s_2,s_3)} \leq 0$ and equality if and only if $a_1 = a_2 = a_3, \ b_1 = b_2 = b_3$ and $a_i + b_{i+2} = 1$, for all $i = 1,2,3$.

        \item $\min{(t_1,t_2,t_3)} < 0$

        \item $\min{(u_1,u_2,u_3)} < 0$\\
        \newline
        For all $i = 1,2,3$ we have:

        \item $\min_{i}{(v_i, s_{i+1}, s_{i+2})} < 0$

        \item $\min_{i}{(s_i, t_{i+1})} < 0$

        \item $\min_{i}{(s_i, u_{i+2})} < 0$

        \item $\min_{i}{(v_i, u_{i+2})} < 0$

        \item $\min_{i}{(v_i, t_{i+1})} < 0$

        \item $\min_{i}{(u_i, t_{i+1})} < 0$
     \end{enumerate}
\end{lemma}

\begin{proof}
ad1. $$3\min{(s_1,s_2,s_3)} \leq s_1+s_2+s_3 = $$
$$= 3 - \left( \frac{a_2}{a_1} + \frac{a_3}{a_2} + \frac{a_1}{a_3} \right) - \left( \frac{b_3}{b_1} + \frac{b_1}{b_2} + \frac{b_2}{b_3} \right) + (a_1 + b_3) + (a_2 + b_1) + (a_3 + b_2) \leq$$
$$\leq 6 - \left( \frac{a_2}{a_1} + \frac{a_3}{a_2} + \frac{a_1}{a_3} \right) - \left( \frac{b_3}{b_1} + \frac{b_1}{b_2} + \frac{b_2}{b_3} \right) \leq$$
$$\leq 6 - 3 - 3 = 0,$$
by the geometric-arithmetric inequality.\\
Moreover we have equality if and only if $a_1 + b_3 = a_2 + b_1 = a_3 + b_2 = 1$ and $\frac{a_2}{a_1} = \frac{a_3}{a_2} = \frac{a_1}{a_3}$ and $\frac{b_3}{b_1} = \frac{b_1}{b_2} = \frac{b_2}{b_3}$ i.e. $a_1 + b_3 = a_2 + b_1 = a_3 + b_2$, \\
$a_1 = a_2 = a_3$ and $b_1 = b_2 = b_3$.\\
\newline
\noindent ad2. $$3\min{(t_1,t_2,t_3)} \leq t_1+t_2+t_3 =$$
$$3 - \frac{a_3}{b_1} - \frac{a_1}{b_2} - \frac{a_2}{b_3} + a_1 + a_2 + a_3 - \left( \frac{a_3}{a_1} + \frac{a_1}{a_2} + \frac{a_2}{a_3} \right) - \left( \frac{a_2}{a_1} + \frac{a_3}{a_2} + \frac{a_1}{a_3} \right) \leq$$
$$\leq -3 - a_1 \left(\frac{1}{b_2} - 1 \right) - a_2 \left(\frac{1}{b_3} - 1 \right) - a_3 \left(\frac{1}{b_1} - 1 \right) < 0,$$
again by the geometric-arithmetric inequality.\\
\newline
\noindent ad3. $$3\min{(u_1,u_2,u_3)} \leq u_1+u_2+u_3 \leq$$
$$\leq -3 - b_1 \left(\frac{1}{a_3} - 1 \right) - b_2 \left(\frac{1}{a_1} - 1 \right) - b_3 \left(\frac{1}{a_2} - 1 \right) < 0, \text{ as in 2.}$$
\newline
\noindent
ad4.
\begin{equation*}
    \begin{split}
    s_{i+1}& = - \frac{a_{i+2}}{a_{i+1}} - \frac{b_{i+3}}{b_{i+1}} + a_{i+2} + b_{i+3} + 1 \leq\\
                 & \leq -a_{i+2} - \frac{b_{i+3}}{b_{i+1}} + a_{i+2} + b_{i+3} + 1 =\\
                 & = 1 - b_{i+3} \left(\frac{1}{b_{i+1}} - 1 \right) = 1 - b_i \left(\frac{1}{b_{i+1}} - 1 \right)
    \end{split}
\end{equation*}
\newline
\noindent
\begin{equation*}
    \begin{split}
    s_{i+2}& = - \frac{a_{i+3}}{a_{i+2}} - \frac{b_{i+4}}{b_{i+2}} + a_{i+3} + b_{i+4} + 1 \leq\\
                 & \leq - \frac{a_{i+3}}{a_{i+2}} - b_{i+4} + a_{i+3} + b_{i+4} + 1 =\\
                 & = 1 - a_{i+3} \left(\frac{1}{a_{i+2}} - 1 \right) = 1 - a_i \left(\frac{1}{a_{i+2}} - 1 \right)
    \end{split}
\end{equation*}
\newline
\noindent
If $b_i \left(\frac{1}{b_{i+1}} - 1 \right) > 1$ or $a_i \left(\frac{1}{a_{i+2}} - 1 \right) > 1$ we have $$\min{(v_i,s_{i+1},s_{i+2})} \leq \min{(s_{i+1},s_{i+2})} < 0.$$
\noindent
So assume, that $\frac{1}{b_i} \geq \frac{1}{b_{i+1}} - 1$ and $\frac{1}{a_i} \geq \frac{1}{a_{i+2}} - 1$.\\
Now
\begin{equation*}
    \begin{split}
    v_i& = -\frac{a_{i+2}}{b_i} - \frac{b_{i+1}}{b_i} - \frac{a_{i+2}}{a_i} - \frac{b_{i+1}}{a_i} + 1 =\\
         & = -\frac{1}{b_i}(a_{i+2} + b_{i+1}) - \frac{1}{a_i}(a_{i+2} + b_{i+1}) + 1 \leq\\
         & \leq - \left( \frac{1}{b_{i+1}} - 1 \right)(a_{i+2} + b_{i+1}) - \left( \frac{1}{a_{i+2}} - 1 \right)(a_{i+2} + b_{i+1}) + 1 =\\
         & = - \left( \frac{1}{b_{i+1}} - 1 \right)a_{i+2} - \left( \frac{1}{a_{i+2}} - 1 \right)b_{i+1} + (a_{i+2} + b_{i+1}) - 1 \leq\\
         & \leq - \left( \frac{1}{b_{i+1}} - 1 \right)a_{i+2} - \left( \frac{1}{a_{i+2}} - 1 \right)b_{i+1} < 0.
    \end{split}
\end{equation*}
\newline
\noindent
ad5.
\begin{equation*}
    \begin{split}
    s_i& = -\frac{a_{i+1}}{a_i} - \frac{b_{i+2}}{b_i} + a_{i+1} + b_{i+2} + 1 \leq\\
         & \leq -\frac{a_{i+1}}{a_i} - b_{i+2} + a_{i+1} + b_{i+2} + 1 = \\
         & = - a_{i+1} \left(\frac{1}{a_i} - 1 \right) + 1
    \end{split}
\end{equation*}
\newline
\noindent
If $\frac{1}{a_{i+1}} < \frac{1}{a_i} - 1$ we get $\min{(s_i,t_{i+1})} \leq s_i < 0$, so assume that $\frac{1}{a_{i+1}} \geq \frac{1}{a_i} - 1$.\\
Now
\begin{equation*}
    \begin{split}
    t_{i+1}& = -\frac{a_i}{b_{i+1}} - \frac{a_i}{a_{i+1}} - \frac{a_{i+2}}{a_{i+1}} + a_{i+2} + 1 \leq\\
                 & \leq -\frac{a_i}{b_{i+1}} - a_i \left(\frac{1}{a_i} - 1 \right) - \frac{a_{i+2}}{a_{i+1}} + a_{i+2} + 1 =\\
                 & = - a_i \left(\frac{1}{b_{i+1}} - 1 \right) - a_{i+2} \left(\frac{1}{a_{i+1}} - 1 \right) < 0
    \end{split}
\end{equation*}
\newline
\noindent
ad6.
\begin{equation*}
    \begin{split}
    s_i& = -\frac{a_{i+1}}{a_i} - \frac{b_{i+2}}{b_i} + a_{i+1} + b_{i+2} + 1 \leq\\
         & \leq -a_{i+1} - \frac{b_{i+2}}{b_i} + a_{i+1} + b_{i+2} + 1 = \\
         & - b_{i+2} \left( \frac{1}{b_i} - 1 \right) + 1
    \end{split}
\end{equation*}
\newline
\noindent
If $\frac{1}{b_{i+2}} < \frac{1}{b_i} - 1$ we get $\min{(s_i,u_{i+2})} \leq s_i < 0$, so assume that $\frac{1}{b_{i+2}} \geq \frac{1}{b_i} - 1$. As in 5. we obtain
$$u_{i+2} \leq -b_{i+1} \left( \frac{1}{b_{i+2}} - 1 \right) - b_i \left(\frac{1}{a_{i+2}} - 1 \right) < 0.$$
\newline
\noindent
ad7.
\begin{equation*}
    \begin{split}
    u_{i+2}& = -\frac{b_{i+1}}{b_{i+2}} - \frac{b_i}{b_{i+2}} - \frac{b_i}{a_{i+2}} + b_{i+1} + 1 \leq\\
                 & \leq -b_{i+1} - \frac{b_i}{b_{i+2}} - \frac{b_i}{a_{i+2}} + b_{i+1} + 1 =\\
                 & = -b_i \left(\frac{1}{a_{i+2}} + \frac{1}{b_{i+2}} \right) + 1
    \end{split}
\end{equation*}
\newline
\noindent
If $\frac{1}{b_i} < \frac{1}{a_{i+2}} + \frac{1}{b_{i+2}}$, we get $u_{i+2} < 0$, so assume that $\frac{1}{b_i} \geq \frac{1}{a_{i+2}} + \frac{1}{b_{i+2}}$.\\
Now
\begin{equation*}
    \begin{split}
    v_i& = -\frac{a_{i+2}}{b_i} - \frac{b_{i+1}}{b_i} - \frac{a_{i+2}}{a_i} - \frac{b_{i+1}}{a_i} + 1 \leq\\
         & \leq - \left(\frac{1}{a_{i+2}} + \frac{1}{b_{i+2}} \right) a_{i+2} - \frac{b_{i+1}}{b_i} - \frac{a_{i+2}}{a_i} - \frac{b_{i+1}}{a_i} + 1 =\\
         & = -\frac{a_{i+2}}{b_{i+2}} - \frac{b_{i+1}}{b_i} - \frac{a_{i+2}}{a_i} - \frac{b_{i+1}}{a_i} < 0.
    \end{split}
\end{equation*}
\newline
\noindent
ad8.
\begin{equation*}
    \begin{split}
    t_{i+1}& = -\frac{a_i}{b_{i+1}} - \frac{a_i}{a_{i+1}} - \frac{a_{i+2}}{a_{i+1}} + a_{i+2} + 1 \leq\\
                 & \leq -\frac{a_i}{b_{i+1}} - \frac{a_i}{a_{i+1}} - a_{i+2} + a_{i+2} + 1 =\\
                 & = -a_i \left( \frac{1}{a_{i+1}} + \frac{1}{b_{i+1}} \right) + 1
    \end{split}
\end{equation*}
\newline
\noindent
If $\frac{1}{a_i} < \frac{1}{a_{i+1}} + \frac{1}{b_{i+1}}$, we get $t_{i+1} < 0$, so assume that $\frac{1}{a_i} \geq \frac{1}{a_{i+1}} + \frac{1}{b_{i+1}}$.\\
As in 7. we obtain
$$v_i \leq -\frac{a_{i+2}}{b_i} - \frac{b_{i+1}}{b_i} - \frac{a_{i+2}}{a_i} - \frac{b_{i+1}}{a_{i+1}} < 0$$
\newline
\noindent
ad9.
\begin{equation*}
    \begin{split}
    u_i& = -\frac{b_{i+2}}{b_i} - \frac{b_{i+1}}{b_i} - \frac{b_{i+1}}{a_i} + b_{i+2} + 1 \leq\\
         & \leq -b_{i+2} - b_{i+1} - \frac{b_{i+1}}{a_i} + b_{i+2} + 1 =\\
         & = -b_{i+1} \left(\frac{1}{a_i} + 1 \right) + 1
    \end{split}
\end{equation*}
\begin{equation*}
    \begin{split}
    t_{i+1}& = -\frac{a_i}{b_{i+1}} - \frac{a_i}{a_{i+1}} - \frac{a_{i+2}}{a_{i+1}} + a_{i+2} + 1 \leq\\
                 & \leq -\frac{a_i}{b_{i+1}} - a_i - a_{i+2} + a_{i+2} + 1 =\\
                 & = - a_i \left(\frac{1}{b_{i+1}} + 1 \right) + 1
    \end{split}
\end{equation*}
Assume that $\min{(u_i, t_{i+1})} \geq 0$. Then we get
$$1 > (1 - b_{i+1})(1 - a_i) \geq \frac{b_{i+1}}{a_i} \frac{a_i}{b_{i+1}} = 1,$$
a contradiction.
\end{proof}

\begin{lemma}
Let $a_1, a_2, a_3, b_1, b_2, b_3$ be real numbers and set $a_4 = a_1, a_5 = a_2, b_4 = b_1$ and $b_5 = b_2$. Assume, that $0 < a_i < 1$, $0 < b_i < 1$
and $a_i+b_{i+2}\leq 1$, for all $i = 1,2,3$. Further for $i = 1,2,3$ let
\begin{eqnarray*}
\alpha_i = \frac{1}{a_i+b_i-a_ib_i}(a_i+b_i+(1-b_{i+2})a_i(1-b_i))
\end{eqnarray*}
\begin{eqnarray*}
\overline{\alpha_i} = \frac{1}{a_i+b_i-a_ib_i}((a_i+b_i)(1-a_{i+2})+a_i(1-b_i))
\end{eqnarray*}
\begin{eqnarray*}
\beta_i = \frac{1}{a_i+b_i-a_ib_i}(a_i+b_i+(1-a_{i+1})b_i(1-a_i))
\end{eqnarray*}
\begin{eqnarray*}
\overline{\beta_i} = \frac{1}{a_i+b_i-a_ib_i}((a_i+b_i)(1-b_{i+1})+b_i(1-a_i))
\end{eqnarray*}
\begin{eqnarray*}
(\alpha_4 = \alpha_1, \overline{\alpha_4}=\overline{\alpha_1},
\beta_4 = \beta_1, \overline{\beta_4}=\overline{\beta_1})
\end{eqnarray*}
and
\begin{eqnarray*}
M_i = \max(\alpha_i,\overline{\alpha_i})+\max(\beta_i,\overline{\beta_i}).
\end{eqnarray*}
Then we have
\begin{eqnarray*}
\min_{i} M_i \leq 3,
\end{eqnarray*}
and equality if and only if $a_1 = a_2 = a_3$,$b_1 = b_2 = b_3$ and $a_i+b_{i+2} = 1$, for all $i = 1,2,3$.\\
\end{lemma}

\begin{proof}
It is easy to see, that $a_1 = a_2 = a_3$,$b_1 = b_2 = b_3$ and
$a_i+b_{i+2} = 1$, for all $i = 1,2,3$, implies $\min_{i} M_i=3$.
\\
For $i = 1,2,3$ let
\begin{eqnarray*}
\epsilon_i = \left\{\begin{array}{ll} 1, & \alpha_i < \overline{\alpha_i} \\
0, & \alpha_i \geq \overline{\alpha_i}\end{array}\right.
\end{eqnarray*}
and
\begin{eqnarray*}
\delta_i = \left\{\begin{array}{ll} 1, & \beta_i < \overline{\beta_i} \\
0, & \beta_i \geq \overline{\beta_i}\end{array}\right.
\end{eqnarray*}
\begin{eqnarray*}
(\epsilon_4 = \epsilon_1,\delta_4 = \delta_1).
\end{eqnarray*}
The definition of $M_i$ leads to several cases, depending on the
values of $\epsilon_i$ and $\delta_i$. In the sequel it is
convenient to define, for given values of $\epsilon_i$ and
$\delta_i$, the corresponding case-vector $c$ by $c =
(\epsilon_1,\delta_1,\epsilon_2,\delta_2,\epsilon_3,\delta_3)$.\\\newline
So for example the case $\alpha_1 < \overline{\alpha_1}$, $\beta_1
\geq \overline{\beta_1}$, $\alpha_2 < \overline{\alpha_2}$,\\
$\beta_2 < \overline{\beta_2}$,
$\alpha_3 < \overline{\alpha_3}$, $\beta_3 \geq \overline{\beta_3}$ is given by $c = (1,0,1,1,1,0)$.\\
Furthermore a vector $C$ in $\left\{0,1,x\right\}^6$ abbreviates the set of cases $c$ in $C$, such that the entries of $c$ and $C$ coincide in all entries unequal
to $x$.\\\newline
So for example $C = (1,0,x,1,0,x)$ is the set of cases $c = (\epsilon_1,\delta_1,\epsilon_2,\delta_2,\epsilon_3,\delta_3)$, such that $\epsilon_1 = 1$, $\delta_1 = 0$,
$\delta_2 = 1$ and $\epsilon_3 = 0$.\\\newline
Assume that $\epsilon_i = \delta_{i+1} = 1$ for some $i = 1,2,3$.  Hence  $\alpha_i < \overline{\alpha_i}$ and $\beta_{i+1} < \overline{\beta_{i+1}}$. \\
$\alpha_i < \overline{\alpha_i}$ implies
\begin{eqnarray*}
\frac{a_i+b_i}{a_i(1-b_i)} < \frac{b_{i+2}}{a_{i+2}},
\end{eqnarray*}
and $\beta_{i+1} < \overline{\beta_{i+1}}$ implies
\begin{eqnarray*}
\frac{a_{i+1}+b_{i+1}}{b_{i+1}(1-a_{i+1})} < \frac{a_{i+2}}{b_{i+2}}.
\end{eqnarray*}
But $\frac{a_i+b_i}{a_i(1-b_i)}$ and $\frac{a_{i+1}+b_{i+1}}{b_{i+1}(1-a_{i+1})}$ are greater then 1 and therefore we would get $a_{i+2} < b_{i+2}$ and $b_{i+2} < a_{i+2}$,
a contradiction.\\\newline
Therefore we have shown, that
\begin{eqnarray*}
C = (1,x,x,1,x,x) = (x,1,x,x,1,x) = (x,x,1,x,x,1) = \emptyset.
\end{eqnarray*}
For further discussion let $s_i$, $t_i$, $u_i$ and $v_i$ $(i = 1,2,3)$ be defined as in Lemma 4.3. It is easy to see, that
\begin{eqnarray*}
\alpha_i + \beta_i - 3 = \frac{a_i b_i}{a_i + b_i - a_i b_i} s_i,
\end{eqnarray*}
\begin{eqnarray*}
\overline{\alpha_i} + \beta_i - 3 = \frac{a_i b_i}{a_i + b_i - a_i b_i} t_i,
\end{eqnarray*}
\begin{eqnarray*}
\alpha_i + \overline{\beta_i} - 3 = \frac{a_i b_i}{a_i + b_i - a_i b_i} u_i,
\end{eqnarray*}
\begin{eqnarray*}
\overline{\alpha_i} + \overline{\beta_i} - 3 = \frac{a_i b_i}{a_i + b_i - a_i b_i} v_i.
\end{eqnarray*}
We have to show, that
\begin{eqnarray*}
\min_{i} \max(s_i, t_i, u_i, v_i) \leq 0
\end{eqnarray*}
and equality if and only if if $a_1 = a_2 = a_3$,$b_1 = b_2 = b_3$ and $a_i+b_{i+2} = 1$, for all $i = 1,2,3$.\\\newline
Interpretating Lemma 4.3 by means of case-vectors were are done in each of the following cases:
\begin{enumerate}
\item $C_1 = \left\{(0,0,0,0,0,0)\right\}$
\item $C_2 = \left\{(1,0,1,0,1,0)\right\}$
\item $C_3 = \left\{(0,1,0,1,0,1)\right\}$
\item $C_4 = \left\{(1,1,0,0,0,0),(0,0,1,1,0,0),(0,0,0,0,1,1)\right\}$
\item $C_5 = (0,0,1,0,x,x)\cup(x,x,0,0,1,0)\cup(1,0,x,x,0,0)$
\item $C_6 = (0,0,x,x,0,1)\cup(0,1,0,0,x,x)\cup(x,x,0,1,0,0)$
\item $C_7 = (1,1,x,x,0,1)\cup(0,1,1,1,x,x)\cup(x,x,0,1,1,1)$
\item $C_8 = (1,1,1,0,x,x)\cup(x,x,1,1,1,0)\cup(1,0,x,x,1,1)$
\item $C_9 = (0,1,1,0,x,x)\cup(x,x,0,1,1,0)\cup(1,0,x,x,0,1)$
\end{enumerate}
To continue let $\widetilde{C}_k$ be the set of all cases, where the number of 1's in the correesponding case vectors is exactly
$k$ ($k \in \left\{0,1, \ldots,6\right\}$).\\\newline
It is a routine to check, that
\begin{eqnarray*}
\widetilde{C}_0 = C_1, \widetilde{C}_1 \subseteq C_5\cup C_6, \widetilde{C}_2 \subseteq C_4\cup C_5\cup C_6\cup C_9
\end{eqnarray*}
and
\begin{eqnarray*}
\widetilde{C}_3 \subseteq C_2\cup C_3\cup C_5\cup C_6\cup C_7\cup C_8\cup C_9.
\end{eqnarray*}
As noted before, we have
\begin{eqnarray*}
(1,x,x,1,x,x) = (x,1,x,x,1,x) = (x,x,1,x,x,1) = \emptyset,
\end{eqnarray*}
which implies
\begin{eqnarray*}
\widetilde{C}_5\cup \widetilde{C}_5\cup \widetilde{C}_6 = \emptyset.
\end{eqnarray*}
Summing up we are done by Lemma 4.3.\\\newline
\end{proof}
\begin{proof}[Proof of Theorem 3.1]
For $i = 1,2,3$ choose points $z_i$ in $S$, such that
\begin{eqnarray*}
\max_{x \in S} \left\|x-v_i\right\|+\left\|x+v_i\right\| =
\left\|z_i-v_i\right\|+\left\|z_i+v_i\right\|.
\end{eqnarray*}
W.l.o.g. let $z_1$ be in $\left[v_1,v_2\right]\cup
\left(v_2,v_3\right)\cup \left[v_3, -v_1\right]$.
If $z_1 \in \left[v_1,v_2\right]\cup \left[v_3, -v_1\right]$ we are done by Lemma 4.2.\\
Hence we can assume, that $z_1 \in \left(v_2,v_3\right)$. The same argument leads to $z_2 \in \left(v_3,-v_1\right)$ and
$z_3 \in \left(-v_1,-v_2\right)$.\\\newline
Therefore there are unique real mubers $x_i, y_i$ $(i = 1,2,3)$, such that\\ $0 < x_i \leq 1$, $0 < y_i \leq 1$, $x_i + y_i \geq 1$ and
\begin{eqnarray*}
z_1 = x_1 v_2 + y_1 v_3,
z_2 = x_2 v_3 + y_2 (-v_1),
z_3 = x_3 (-v_1) + y_3 (-v_2)
\end{eqnarray*}
Now let $B_0$ be the convex hull of $\left\{\pm v_1,\pm v_2,\pm v_3,\pm z_1,\pm z_2,\pm z_3\right\}$. $B_0$ defines a norm $\left\|.\right\|_0$,
such that $B_0 = \left\{x \in E, \left\|x\right\|_0 \leq 1\right\}$.\\
Since $B_0 \subseteq \left\{x \in E, \left\|x\right\| \leq
1\right\}$, we obtain $\left\|x\right\| \leq \left\|x\right\|_0$,
for all $x$  in $E$ and therefore we are done, if we can show
\begin{eqnarray*}
\min_{i} \left\|z_i-v_i\right\|_0+\left\|z_i+v_i\right\|_0 \leq 3.
\end{eqnarray*}
Routine calculations lead to
\begin{eqnarray*}
\left\|z_i-v_i\right\|_0 = \max(x_i+y_i+\frac{1-x_{i+1}}{y_{i+1}}(1-x_i),(x_i+y_i)\frac{1-y_{i+1}}{x_{i+1}}+1-x_i),
\end{eqnarray*}
\begin{eqnarray*}
\left\|z_i+v_i\right\|_0 = \max(x_i+y_i+\frac{1-y_{i+2}}{x_{i+2}}(1-y_i),(x_i+y_i)\frac{1-x_{i+2}}{y_{i+2}}+1-y_i),
\end{eqnarray*}
for $i = 1,2,3$, with $x_4 = x_1$,$x_5 = x_2$,$y_4 = y_1$ and $y_5 = y_2$.\\\newline
Let $H_1$ be the closed halfspace defined by the line through $v_3$ and $z_1$, such that $0 \in H_1$. Since $z_2 \in H_1$ we get
\begin{eqnarray*}
\frac{1-x_2}{y_2} + \frac{1-y_1}{x_1} \geq 1.
\end{eqnarray*}
Note that if $H_2$ denotes the closed halfspace defined by line through $v_3$ and $z_2$, such that $0 \in H_2$, we have $z_2 \in H_1$ if and only if $z_1 \in H_2$.
Hence $z_1 \in H_2$ again leads to
\begin{eqnarray*}
\frac{1-x_2}{y_2} + \frac{1-y_1}{x_1} \geq 1.
\end{eqnarray*}
The same argument (looking at $(-v_1)$ and $(-v_2)$) implies
\begin{eqnarray*}
\frac{1-x_3}{y_3} + \frac{1-y_2}{x_2} \geq 1
\end{eqnarray*}
and
\begin{eqnarray*}
\frac{1-x_1}{y_1} + \frac{1-y_3}{x_3} \geq 1.
\end{eqnarray*}
Assume that $x_1+y_1 = 1$. Therefore $z_1 \in \overline{v_2 v_3}$ and so
\begin{eqnarray*}
\left\|z_1-v_1\right\|_0 +\left\|z_1+v_1\right\|_0 \leq& \max(\left\|v_2-v_1\right\|_0 +\left\|v_2+v_1\right\|_0, \left\|v_3-v_1\right\|_0 +\left\|v_3+v_1\right\|_0)\\=&
\max(1 + \left\|v_2+v_1\right\|_0, \left\|v_3-v_1\right\|_0 + 1)\\\leq& 3.
\end{eqnarray*}
The same argument shows, that $x_i+y_i = 1$ implies
\begin{eqnarray*}
\left\|z_i-v_i\right\|_0+\left\|z_i+v_i\right\|_0 \leq 3,
\end{eqnarray*}
for all $i = 1,2,3$.\\\newline Furthermore assume that $x_1 = 1$ or
$y_1 = 1$.  Since
\begin{eqnarray*}
\frac{1-x_1}{y_1} + \frac{1-y_3}{x_3} \geq 1, \frac{1-y_1}{x_1} + \frac{1-x_2}{y_2} \geq 1
\end{eqnarray*}
and $x_3+y_3 \geq 1$, $x_2+y_2 \geq 1$, we would get $x_3+y_3 = 1$ or $x_2+y_2 = 1$ and hence
\begin{eqnarray*}
\min_{i} \left\|z_i-v_i\right\|_0+\left\|z_i+v_i\right\|_0 \leq 3,
\end{eqnarray*}
as mentioned above.\\
The same argument shows, that $x_2 = 1$ or $y_2 = 1$ or $x_3 = 1$ or $y_3 = 1$ leads to
\begin{eqnarray*}
\min_{i} \left\|z_i-v_i\right\|_0+\left\|z_i+v_i\right\|_0 \leq 3.
\end{eqnarray*}
Summing up it remains to show, that for real numbers $x_i, y_i$ ($i = 1,2,3$), $x_4 = x_1$, $x_5 = x_2$, $y_4 = y_1$,$y_5 = y_2$ with
$0 < x_i < 1$, $0 < y_i < 1$, $x_i + y_i > 1$ and
\begin{eqnarray*}
\frac{1-x_i}{y_i} + \frac{1-y_{i+2}}{x_{i+2}} \geq 1,
\end{eqnarray*}
for all $i = 1,2,3$, we have
\begin{eqnarray*}
\min_i M_i \leq 3,
\end{eqnarray*}
where
\begin{eqnarray*}
M_i = \max (\beta_i, \overline{\beta_i}) + \max (\alpha_i, \overline{\alpha_i})
\end{eqnarray*}
and
\begin{eqnarray*}
\alpha_i = x_i+y_i+\frac{1-y_{i+2}}{x_{i+2}}(1-y_i)
\end{eqnarray*}
\begin{eqnarray*}
\overline{\alpha_i} = (x_i+y_i)\frac{1-x_{i+2}}{y_{i+2}}+1-y_i
\end{eqnarray*}
\begin{eqnarray*}
\beta_i = x_i+y_i+\frac{1-x_{i+1}}{y_{i+1}}(1-x_i)
\end{eqnarray*}
\begin{eqnarray*}
\overline{\beta_i} = (x_i+y_i)\frac{1-y_{i+1}}{x_{i+1}}1-x_i,
\end{eqnarray*}
for $i = 1,2,3$.\\\newline
Finally for $i = 1,2,3$ set
\begin{eqnarray*}
a_i = 1 - \frac{1 - x_i}{y_i}, b_i = 1 - \frac{1 - y_i}{x_i},
\end{eqnarray*}
$a_4 = a_1$, $a_5 = a_2$, $b_4 = b_1$ and $b_5 = b_2$. \\
It follows, that $0 < a_i < 1$, $0 < b_i < 1$ and $a_i + b_{i+2} \leq 1$, for all $i = 1,2,3$.\\
It is easy to check, that
\begin{eqnarray*}
\alpha_i = \frac{1}{a_i+b_i-a_ib_i}(a_i+b_i+(1-b_{i+2})a_i(1-b_i))
\end{eqnarray*}
\begin{eqnarray*}
\overline{\alpha_i} = \frac{1}{a_i+b_i-a_ib_i}((a_i+b_i)(1-a_{i+2})+a_i(1-b_i))
\end{eqnarray*}
\begin{eqnarray*}
\beta_i = \frac{1}{a_i+b_i-a_ib_i}(a_i+b_i+(1-a_{i+1})b_i(1-a_i))
\end{eqnarray*}
\begin{eqnarray*}
\overline{\beta_i} = \frac{1}{a_i+b_i-a_ib_i}((a_i+b_i)(1-b_{i+1})+b_i(1-a_i))
\end{eqnarray*}
for $i = 1,2,3$. \\\newline
Applying Lemma 4.4 we obtain
\begin{eqnarray*}
\min_i M_i \leq 3
\end{eqnarray*}
and hence we are done.\\\newline
\end{proof}
\begin{proof}[Proof of Theorem 3.2]
Consider an affine regular hexagon with vertex  set $H = \left\{\pm v_1, \pm v_2, \pm v_3\right\} \subseteq S$ inscribted to $S$. Applying Theorem 3.1, we get
\begin{eqnarray*}
\min_{y \in S}\max_{x \in S} \left\|x - y\right\|+\left\|x + y\right\| \leq \min_i \max_{x \in S} \left\|x - v_i\right\|+\left\|x + v_i\right\| \leq 3.
\end{eqnarray*}
As noted in section 1, we have
\begin{eqnarray*}
\min_{y \in S}\max_{x \in S} \left\|x - y\right\|+\left\|x + y\right\| = 3,
\end{eqnarray*}
if $S$ is a parallelogram or an affine regular hexagon. Now assume,
that
\begin{eqnarray*}
\min_{y \in S}\max_{x \in S} \left\|x - y\right\|+\left\|x + y\right\| = 3.
\end{eqnarray*}
\begin{enumerate}
\item All three arcs $\left[v_1, v_2\right]$, $\left[v_2, v_3\right]$ and $\left[v_3, -v_1\right]$ are line segments:\\
            therefore $S$ is an affine regular hexagon with vertex set $\left\{\pm v_1, \pm v_2, \pm v_3\right\}$.
\item Exactly two of the three arcs $\left[v_1, v_2\right]$, $\left[v_2, v_3\right]$ and $\left[v_3, -v_1\right]$ are line segments:\\
            w.l.o.g. let $\left[v_2, v_3\right] = \overline{v_2 v_3}$  and $\left[v_3, -v_1\right] = \overline{v_3 (-v_1)}$ and $\left[v_1, v_2\right] \neq \overline{v_1 v_2}$.\\
            For $y$ in $S$ let
            \begin{eqnarray*}
            \alpha(y) = \max_{x \in \left[v_1, v_2\right]}\left\|x - y\right\|+\left\|x + y\right\|
            \end{eqnarray*}
            By convexity we have
            \begin{eqnarray*}
            \max_{x \in S}\left\|x - v_1\right\|+\left\|x + v_1\right\| = \max(\left\|v_3 - v_1\right\|+\left\|v_3 + v_1\right\|,\alpha(v_1))
            \end{eqnarray*}
            By Lemma 4.2 we have $\alpha(v_1) \leq 3$. \\\newline
            If $\alpha(v_1) < 3$, we can choose some point $u \neq v_1$ in $\overline{(-v_3)v_1}$ close to $v_1$, such that $\alpha(u) < 3$.\\
            Since $\left\|v_3 + u\right\| < 1$, we get
            \begin{eqnarray*}
            \max_{x \in S}\left\|x - u\right\|+\left\|x + u\right\| = \max(\left\|v_3 - u\right\|+\left\|v_3 + u\right\|,\alpha(u)) < 3,
            \end{eqnarray*}
            a contradiction to
            \begin{eqnarray*}
            \max_{x \in S} \left\|x - u\right\|+\left\|x + u\right\| \geq 3.
            \end{eqnarray*}
            If $\alpha(v_1) = 3$, Lemma 4.2 again implies either $\left[v_1, v_2\right] = \overline{v_1 v_2}$ or there is some $z \in \left(v_1, v_2\right)$ such that
            $\left[v_1, z\right] = \overline{v_1 z}$ and $\left[z, v_3\right] = \overline{z v_3}$.\\\newline
            Since by assumption $\left[v_1, v_2\right] \neq \overline{v_1 v_2}$ we can choose some $z \in \left(v_1, v_2\right)$, such that
            $\left[v_1, z\right] = \overline{v_1 z}$ and $\left[z, v_3\right] = \overline{z v_3}$. \\
            Of course, we have $\left\|z - v_2\right\| \leq 1$. $\left\|z - v_2\right\| = 1$ implies $z = v_1 + v_2$ and $S$ is a parallelogram with vertex set
            $\left\{\pm v_3, \pm z\right\}$.\\
            So assume, that $\left\|z - v_2\right\| < 1$: \\\newline
            By convexity we have
            \begin{eqnarray*}
            \max_{x \in S} \left\|x - v_2\right\|+\left\|x + v_2\right\| = \max_{x \in \left\{v_1, z, v_3\right\}} \left\|x - v_2\right\|+\left\|x + v_2\right\|.
            \end{eqnarray*}
            Since $\left\|z - v_2\right\|+\left\|z + v_2\right\| < 3$\\
            and $\left\|v_1 - v_2\right\|+\left\|v_1 + v_2\right\| < \left\|v_1 - v_2\right\| + 2 = 3$, we can
            choose some point $w \neq v_2$ in $\overline{v_2 v_3}$ close to $v_2$, such that $\left\|z - w\right\|+\left\|z + w\right\| < 3$ and
            $\left\|v_1 - w\right\|+\left\|v_1 + w\right\| < 3$. But $\left\|v_3 - w\right\|+\left\|v_3 + w\right\| < 1 + \left\|v_3 + w\right\| \leq 3$, and therefore
            \begin{eqnarray*}
            \max_{x \in S} \left\|x - w\right\|+\left\|x + w\right\| < 3,
            \end{eqnarray*}
            a contradiciton to
            \begin{eqnarray*}
            \max_{x \in S} \left\|x - w\right\|+\left\|x + w\right\| \geq 3.
            \end{eqnarray*}
\item Exactly one of the three arcs $\left[v_1, v_2\right]$, $\left[v_2, v_3\right]$ and $\left[v_3, -v_1\right]$ is a line segment:\\
            w.l.o.g. let $\left[v_2, v_3\right] = \overline{v_2 v_3}$, $\left[v_1, v_2\right] \neq \overline{v_1 v_2}$ and $\left[v_3, -v_1\right] \neq \overline{v_3 (-v_1)}$.\\
            By convexity we have
            \begin{eqnarray*}
            \max_{x \in S} \left\|x - v_1\right\|+\left\|x + v_1\right\| = \max_{x \in \left[v_1, v_2\right]\cup \left[v_3, -v_1\right]} \left\|x - v_1\right\|+\left\|x + v_1\right\| \leq 3,
            \end{eqnarray*}
            by Lemma 4.2. Since
            \begin{eqnarray*}
            \max_{x \in S} \left\|x - v_1\right\|+\left\|x + v_1\right\| \geq 3,
            \end{eqnarray*}
            we get
            \begin{eqnarray*}
            \max_{x \in \left[v_1, v_2\right]\cup \left[v_3, -v_1\right]} \left\|x - v_1\right\|+\left\|x + v_1\right\| = 3.
            \end{eqnarray*}
            Therefore we can choose some $z \in \left[v_1, v_2\right]\cup \left[v_3, -v_1\right]$, such that \\$\left\|z - v_1\right\|+\left\|z + v_1\right\| = 3$.\\\newline
            W.l.o.g. let $z \in \left[v_1, v_2\right]$:\\\newline
            If $z = v_2$, Lemma 4.2 implies $\left[v_1, v_2\right] = \overline{v_1 v_2}$, a contradiction to \\$\left[v_1, v_2\right] \neq \overline{v_1 v_2}$.\\\newline
            If $z \in\left(v_1, v_2\right)$, Lemma 4.2 again implies $\left[v_1, z\right] = \overline{v_1 z}$ and $\left[z, v_3\right] = \overline{z v_3}$. Since
            $\left\|z - v_3\right\| > \left\|v_2 - v_3\right\| = 1$, we can find $w \in \overline{v_2 v_3}\setminus\left\{v_3\right\}$, such that $\left\|z - w\right\| = 1.$\\\newline
            Since $w = z-v_1$, $H' = \left\{\pm v_1, \pm z, \pm w\right\}$ is the vertex set of an affine regular hexagon with $\left[v_1, z\right] = \overline{v_1 z}$,\\
            $\left[z, w\right] = \overline{z w}$ and $\left[w, -v_1\right] \neq \overline{w (-v_1)}$ and therefore we are done by case 2 (the affine regular hexagon at the beginning
            of the proof was chosen arbitrarily).\\
\item None of the three arcs $\left[v_1, v_2\right]$, $\left[v_2, v_3\right]$ and $\left[v_3, -v_1\right]$ is a line segment:\\
            Let
            \begin{eqnarray*}
            \alpha_1 = \max_{x \in \left[v_1, v_2\right]\cup \left[v_3, -v_1\right]} \left\|x - v_1\right\|+\left\|x + v_1\right\|,
            \end{eqnarray*}
            \begin{eqnarray*}
            \alpha_2 = \max_{x \in \left[v_2, v_3\right]\cup \left[-v_1, -v_2\right]} \left\|x - v_2\right\|+\left\|x + v_2\right\|,
            \end{eqnarray*}
            \begin{eqnarray*}
            \alpha_3 = \max_{x \in \left[v_3, -v_1\right]\cup \left[-v_2, -v_3\right]} \left\|x - v_3\right\|+\left\|x + v_3\right\|.
            \end{eqnarray*}
            By Lemma 4.2 we have
            \begin{eqnarray*}
            \max (\alpha_1, \alpha_2, \alpha_3) \leq 3.
            \end{eqnarray*}
            If $\max (\alpha_1, \alpha_2, \alpha_3) = 3$, let w.l.o.g. $\alpha_1 = 3$. \\
            Again applying Lemma 4.2 we get $\left[v_1, v_2\right] = \overline{v_1 v_2}$ or $\left[v_2, v_3\right] = \overline{v_2 v_3}$ or
            $\left[v_3, -v_1\right] = \overline{v_3 (-v_1)}$, a contradiction to $\left[v_1, v_2\right] \neq \overline{v_1 v_2}$,
            $\left[v_2, v_3\right] \neq \overline{v_2 v_3}$ and $\left[v_3, -v_1\right] \neq \overline{v_3 (-v_1)}$.\\\newline
            Hence we can assume that
            \begin{eqnarray*}
            \max (\alpha_1, \alpha_2, \alpha_3) < 3.
            \end{eqnarray*}
            By assumption we have
            \begin{eqnarray*}
            \min_{i}\max_{x \in S} \left\|x - v_i\right\|+\left\|x + v_i\right\| \geq 3.
            \end{eqnarray*}
            Summing up, we can choose $z_1 \in \left(v_2, v_3\right)$, $z_2 \in \left(v_3, -v_1\right)$ and\\
            $z_3 \in \left(-v_1, -v_2\right)$ such that $\left\|z_i - v_i\right\|+\left\|z_i + v_i\right\| \geq 3$, for all $i = 1,2,3$ .\\\newline
            As in the proof of Theorem 3.1 let $B_0$ be the convex hull of \\
            $\left\{\pm v_1, \pm v_2, \pm v_3, \pm z_1, \pm z_2, \pm z_3\right\}$.\\\newline
            $B_0$ defines a norm $\left\|.\right\|_0$, such that $B_0 = \left\{x \in E, \left\|x\right\|_0 \leq 1\right\}$ and\\
            $\left\|x\right\| \leq \left\|x\right\|_0$, for all $x$ in $E$.\\\newline
            The proof of Theorem 3.1 leads to
            \begin{eqnarray*}
            \min_{i} M_i \leq 3,
            \end{eqnarray*}
            where
            \begin{eqnarray*}
            M_i = \left\|z_i  - v_i\right\|_0 +\left\|z_i + v_i\right\|_0,
            \end{eqnarray*}
            for $i = 1,2,3$.\\
            Since
            \begin{eqnarray*}
            3 \leq \min_i \left\|z_i - v_i\right\| +\left\|z_i + v_i\right\| \leq \min_i M_i \leq 3
            \end{eqnarray*}
            we get
            \begin{eqnarray*}
            \min_{i} M_i = 3.
            \end{eqnarray*}
            As in the proof of Theorem 3.1, let \\
            $z_1 = x_1 v_2 + y_1 v_3$, $z_2 = x_2 v_3 + y_2 (-v_1)$, $z_3 = x_3 (-v_1) + y_3 (-v_2)$,
            with $0 < x_i \leq 1$, $0 < y_i \leq 1$, $x_i + y_i \geq 1$ and $\frac{1 - x_i}{y_i} + \frac{1 - y_{i + 2}}{x_{i + 2}} \geq 1$, for all $i = 1,2,3$
            ($x_4 = x_1$, $x_5 = x_2$, $y_4 = y_1$, $y_5 = y_2$). Since none of the arcs $\left[v_1, v_2\right]$,
            $\left[v_2, v_3\right]$, $\ldots$, $\left[-v_3, v_1\right]$ are line segments we get: \\
            $0 < x_i < 1$, $0 < y_i < 1$ and $x_i + y_i > 1$. Again for $i = 1,2,3$ set
            \begin{eqnarray*}
            a_i = 1 - \frac{1 - x_i}{y_i},
            \end{eqnarray*}
        \begin{eqnarray*}
            b_i = 1 - \frac{1 - y_i}{x_i},
            \end{eqnarray*}
            $a_4 = a_1, a_5 = a_2, b_4 = b_1$ and $b_5 = b_2$
          ($0 < a_i < 1$, $0 < b_i < 1$, $a_i + b_{i+2} \geq 1$, for all $i = 1,2,3$).\\\newline
          By Lemma 4.4 and
          \begin{eqnarray*}
            \min_i M_i = 3
            \end{eqnarray*}
            we obtain $a_1 = a_ 2 = a_3$, $b_1 = b_2 = b_3$, $a_1 + b_3 = 1$, $a_2 + b_1 = 1$ and $a_3 + b_2 = 1$.\\\newline
            This implies $x_1 = x_ 2 = x_3$, $y_1 = y_2 = y_3$ and $\frac{1 - x_i}{y_i} + \frac{1 - y_{i + 2}}{x_{i + 2}} = 1$, for all $i = 1,2,3$.\\\newline
            Looking at $v_3$ the equality $\frac{1 - x_2}{y_2} + \frac{1 - y_1}{x_1} = 1$ implies, that the closed halfspaces $H_1$ and $H_2$ defined in the proof
            of Theorem 3.1 coincide. Therefore the line segment $\overline{z_1 z_2} = \overline{z_1 v_3}\cup \overline{v_3 z_2}$ is contained in $S$.\\\newline
            The same argument for $(-v_1)$ and $(-v_2)$ shows that $S$ is a hexagon with vertex set $\left\{\pm z_1, \pm z_2, \pm z_3\right\}$.\\\newline
            Finally let $x = x_1 = x_2 = x_3$ and $y = y_1 = y_2 = y_3$.\\
            Now
            \begin{eqnarray*}
            z_2 - z_1 =& x v_3 + y (-v_1) - x v_2 - y v_3 \\=& (x - y)(v_2 - v_1) - y v_1 - x v_2 \\=& z_3
            \end{eqnarray*}
            and therefore $S$ is an affine regular hexagon with vertex set $\left\{\pm z_1, \pm z_2, \pm z_3\right\}$.
\end{enumerate}
\end{proof}

\end{document}